\documentclass[11pt,english]{amsart}
\usepackage[T1]{fontenc}
\usepackage[T1]{fontenc}
\usepackage{tikz}
\usepackage[all,cmtip]{xy}
\usetikzlibrary{fadings}
\usepackage[a4paper]{geometry}
 \usepackage{framed}
\usepackage{color}
\usepackage{amstext}
\usepackage{amsthm}
\usepackage{amssymb}
\usepackage{varioref}
\usepackage{amssymb}
\usepackage{mathrsfs}
\usepackage{enumerate}
\usepackage{graphicx}
\usepackage{url}
\usepackage{float}
\usepackage{hyperref}
\hypersetup{
    colorlinks,
    citecolor=blue,
    filecolor=black,
    linkcolor=blue,
    urlcolor=black
}
\hypersetup{colorlinks=true}
\makeatletter
\@namedef{subjclassname@2020}{%
  \textup{2020} Mathematics Subject Classification}
\makeatother

\title[Cone length and LS-category]{Cone length and Lusternik-Schnirelmann category in rational homotopy}

\date{\today}

\author{Paul-Eug\`ene Parent}
\address{D\'epartement de Math{\'e}matiques\\
	Complexe STEM, pi\`ece 361\\
         150 Louis-Pasteur Pvt\\
         Universit\'e d'Ottawa\\
         Ottawa, ON, K1N 6N5\\
         Canada}
\email{pparent@uottawa.ca}

\author{Daniel Tanr\'e}
\address{D\'epartement de Math{\'e}matiques\\
         UMR-CNRS 8524 \\
         Universit\'e de Lille\\
         59655 Villeneuve d'Ascq Cedex\\
         France}
\email{Daniel.Tanre@univ-lille.fr}



\makeatletter
\renewcommand\l@subsection{\@tocline{2}{0pt}{2pc}{5pc}{}}
\renewcommand\l@subsubsection{\@tocline{3}{0pt}{4pc}{10pc}{}}
\makeatother

\theoremstyle{plain}

\newtheorem{theoremv}{Main Theorem}

\newtheorem{proposition}{Proposition}[section]
\newtheorem{theoremb}[proposition]{Theorem}

\theoremstyle{definition}
\newtheorem{definition}[proposition]{Definition}

\theoremstyle{remark}
\newtheorem{remark}[proposition]{Remark}

\numberwithin{equation}{section}
\newcommand{\secref}[1]{Section~\ref{#1}}

\newcommand{\thmref}[1]{Theorem~\ref{#1}}
\newcommand{\propref}[1]{Proposition~\ref{#1}}

\newcommand{\remref}[1]{Remark~\ref{#1}}

\newcommand{\defref}[1]{Definition~\ref{#1}}

\def\ov{\overline}

\def\cC{{\mathcal C}}

\def\cJ{{\mathcal J}}

\def\cL{{\mathcal L}}

\def\cS{{\mathcal S}}
\def\cU{{\mathcal U}}

\def\crI{{\mathscr I}}

\def\crK{{\mathscr K}}
\def\crL{{\mathscr L}}

\def\crM{{\mathscr M}}

\def\1{{\mathbf 1}}

\def\ttx{{\mathtt x}}
\def\L{\mathbb{L}}

\def\Q{\mathbb{Q}}

\def\dga{{{\rm DGA}\;}}
\def\cdgc{{{\rm CDGC}\;}}
\def\cdga{{{\rm CDGA}\;}}

\def\dgl{{{\rm DGL}\;}}

\newcounter{ejemplo}

\newcounter{figura}

\def\cat{\mathrm{cat}\,}

\def\cl{{\mathrm{Cl \,}}}

\def\tx{{\tilde{x}}}
\begin{document}

\begin{abstract}
Lusternik-Schnirelmann category (LS-category) of a topological space is the least integer $n$ such that
there is a covering of $X$ by $n+1$ open sets, each of them being contractible in $X$.
The cone-length is the minimum number of iterated cofibrations necessary to get a space in the  homotopy type of $X$, 
starting from a contractible space. 
The LS-category of a space is always less than or equal to its cone-length. 
Moreover, these two invariants differ by at most one. 
In 1981, J.-M. Lemaire and F. Sigrist conjectured that they are always equal for rational spaces. 
This conjecture is clearly true for spaces  of LS-category 1 and, in 1986, Y. F\'elix and J-C. Thomas verified 
it for spaces of LS-category 2.
But, in 1999,  the general conjecture is invalidated by N. Dupont who built a rational space 
of cone-length 4 and LS-category  3.
In this work, we  provide examples of rational spaces
of cone-length $(k+1)$ and LS-category $k$ for any $k>2$.
\end{abstract}

\maketitle

\tableofcontents

\newpage
\section{Introduction and recalls}
\begin{quote}
This section contains  definitions and properties of Lusternik-Schnirelmann category and of cone-length.
For more details on them, we refer to  \cite{CLOT}. 
We also present the objective of this work and place it among known results.
\end{quote}

We write $Y\simeq X$ if $Y$and $X$ are topological  spaces of the same homotopy type.

\begin{definition}\label{def:LS}
The \emph{Lusternik-Schnirelmann category} (LS-category) of a topological space $X$ is
the least integer $n$ such that
there is a covering of $X$ by $n+1$ open sets, each of them being contractible in $X$.
We denote $\cat X=n$.
\end{definition}

If $X$ is a manifold, its LS-category is a lower bound on the number of critical points for any smooth function
on $X$. This property was the original motivation behind Lusternik and Schnirelmann's definition but it reaches
consequences also in homotopy theory, as the following feature.

\medskip
A space $X$ is a homotopy retract of  $Y$ if there exist maps, $f\colon X\to Y$ and $g\colon Y \to X$,
such that $g\circ f$ is homotopic to the identity. In this situation, we have
$\cat X\leq \cat Y$. In particular, we deduce $\cat X \leq \cat (X\vee S^n)$.
The behavior of $\cat$ towards homotopy retractions makes it a homotopy invariant: if $X\simeq Y$, then $\cat X=\cat Y$.
In contrary if, in \defref{def:LS} we replace ``each of them being contractible in $X$'' 
by ``each of them being contractible'' we lose this invariance (\cite[Proposition 3.11]{CLOT}).
To obtain a homotopy invariant with this variation, we have to consider 
the minimal value on all the spaces in the homotopy type of $X$.
This gives  the definition of strong category. 
We do not continue in this direction, preferring an equivalent formulation called cone-length
(\cite{Ganea, Octav1}).

\begin{definition}\label{def:conelength}
The \emph{cone-length} of a path-connected space, denoted $\cl X$, is 0 if $X$ is contractible and, otherwise, is equal to the
smallest integer $n$ such that there are cofibration sequences
$Z_{i-1}\to Y_{i-1}\to Y_{i}$, $1\leq i\leq n$, with $Y_{0}\simeq *$ and $Y_{n}\simeq X$.
\end{definition}

It is known that a space $X$ is of LS-category one if, and only if, $X$ has a structure of co-H space
 (\cite[Example 1.49]{CLOT}) and of cone-length one if, and only if, it is a suspension
 (\cite[Proposition 3.16]{CLOT}). Thus, 
 a co-H space $X$ which is not a suspension satisfies $\cat X=1$ and $\cl X=2$.
 As such spaces exist, the two invariants do not coincide in general. 
 But they are close, as shows the following result.
 
 \begin{proposition}[{\cite{Takens,Octav1}}]\label{prop:catcl}
 For any path-connected normal ANR, $X$, there is a suspension $\Sigma Z$ such that 
 $\cl(X\vee \Sigma Z)=\cat X$. Moreover,
   $$\cat X\leq\cl X\leq\cat X+1.$$
 \end{proposition}

Let's now enter in the rational world. Here by ``rational space'' we mean
a simply connected space whose homotopy groups are 
rational vector spaces. 
The definition of cone-length  requires some adaptation for which we  follow \cite[Page 836]{Octav2}: 
for rational spaces, in \defref{def:conelength}, the cofibration sequence starts with $i=2$, 
the space $Y_{1}$ is required to be of the homotopy type of
the rationalization of a simply connected suspension and  
the spaces $Z_{i}$  have to be rational and simply connected.

\medskip
With rational spaces, the behaviors of $\cat$ and $\cl$ are  different. First, here,  any co-H space is a  wedge of spheres, 
thus we cannot find examples of  spaces of LS-category one and cone-length two. Moreover,
Y. F\'elix and J.-C.~Thomas (\cite{FT1})  have proved that any rational space of LS-category 2 is of cone-length 2 also.
(See also \cite{Lucia1} for a  short proof for spaces without finite type restriction and  admitting 
  some torsion in their homotopy groups.)
This result reinforces the conjecture of Lemaire and Sigrist 
(\cite{LemaireSigrist}) that the two invariants, $\cat$ and $\cl$, coincide for rational spaces.
In \cite{Dupont}, N. Dupont provides a negative answer with the example of a space of LS-category 3 and cone-length 4. 
Let us also mention the existence of a family of spaces $X_{k}$ of category $k$ and cone-length $(k+1)$,
constructed by D. Stanley in \cite{Don}; these examples use the torsion part of homotopy groups and are not
adapted for a rational consideration.

\medskip
We complete the collection of rational spaces having distinct $\cat$ and $\cl$ invariants.

\begin{theoremv}
For any $k\geq 3$, there exist rational spaces $X_{k}$ such that
$$\cat X_{k}=k\quad \text{and} \quad \cl X_{k}=k+1.$$
\end{theoremv}

  The main pattern of proofs is that of \cite{Dupont}, which corresponds to $k=3$, with arguments adapted to the greater generality of the case treated.
We build the spaces $X_{k}$ as  realizations  of  differential graded Lie algebras $\crL_{k}$
with $\cat\crL_{k}=k$ and $\cl \crL_{k}=k+1$. 
\secref{sec:Lie} contains the necessary presentation of these models for our purpose. 
 In \secref{sec:examplescat}, 
  we begin with a generic case of cell attachment that does not increase the LS-category. 
In \secref{sec:examplescl}, we construct the graded differential Lie algebras $\crL_{k}$ endowed with a decomposition
of length $k+1$, see \defref{def:decomposition}. 
Finally, we show that there are no  isomorphisms of $\crL_{k}$ 
that  shorten the length of decomposition   in \secref{sec:iso}
and finish the proof of the Main Theorem in \secref{sec:proof}.

\medskip
{\sc Notation:} 
In the text, we often abbreviate the expression "differential graded Lie algebra" to dgl
and denote by  \dgl the category of reduced differential graded Lie algebras; i.e., no element of degree 
$q$ with $q<1$.
The \emph{free product} of two dgl's $L$ and $L'$ (which is the sum in the
category \dgl\!\!) is denoted by $L\sqcup L'$.

We represent by $|a|$ the degree of a homogeneous element $a$ of a graded vector space $W$.
The subvector space of elements of degree $i$ is denoted $W_{i}$.
In the case of a dgl $(L,d)$, the space of indecomposables is the quotient
$QL=L/[L,L]$ endowed with the induced differential $Qd$. If $L$ is the graded free Lie algebra, $L=\L(W)$, 
we have $QL=W$ and the graduation by the bracket length is denoted by $\L(W)=\oplus_{i\geq 1}\L^{[i]}(W)$.
This induces a decomposition of the differential $d=d_{1}+ d_{2}+\dots$ with $d_{i}(W)\subset \L^{[i]}(W)$.
An extra filtration on $W$ will be represented by $W=\oplus_{i\geq 1}W_{(i)}$. We also denote
$W_{\leq (i)}=\oplus_{j=1}^iW_{(j)}$ and
$W_{<(i)}=\oplus_{j=1}^{i-1}W_{(j)}$.

\smallskip
For the convenience of the reader, we also have included in \remref{rem:recap}
a summary of the degree and differential of  various elements introduced in the text.
\section{Lie models of spaces}\label{sec:Lie}

\begin{quote}
In this section, we recall the algebraization of rational spaces in terms of differential graded Lie algebras,
made by Quillen in \cite{Quillen}, see also \cite{Daniel} for concrete examples, models of fibrations, ... 
\end{quote}

A (graded) \emph{Lie algebra} consists of a graded vector space, $L=\oplus_{i\geq 1}L_{i}$, 
together with a linear product, called the \emph{Lie bracket},
$$[-,-]\colon L_{i}\otimes L_{j}\to L_{i+j},$$
 verifying the following properties for homogeneous elements,
\begin{eqnarray*}
\text{Antisymmetry:}&
[a,b]=-(-1)^{|a]\,|b|}[b,a],\\
\text{Jacobi identity:}&
(-1)^{|a|\,|c|} [a,[b,c]]+
(-1)^{|a|\,|b|}  [b,[c,a]] +
(-1)^{|b|\,|c|}  [c,[a,b]]=0.
\end{eqnarray*}

A  \emph{differential graded Lie algebra}, henceforth dgl, is a Lie algebra endowed with a differential,
$\partial\colon L_{i}\to L_{i-1}$ such that, for any homogeneous elements $a$, $b$, one has
\begin{equation}\label{equa:derivationLie}
\partial [a,b]=[\partial a,b]+(-1)^{|a|}[a,\partial b].
\end{equation}

Let $W=\oplus_{i\geq 1}W_{i}$ be a graded vector space.
We denote by $\L(W)$ the free graded Lie algebra on $W$.
By  ``free dgl'', we mean  a dgl, free as Lie algebra.
A \emph{minimal} dgl is a free dgl $(\L(W),\partial)$ where the differential $\partial$ takes value in $\L^{[\geq 2]}(W)$. 

\medskip
In \cite{Quillen}, Quillen defines a series of couples of adjoint functors linking the category $\cS_{2}$ of  
simply connected pointed spaces with \dgl\!,
inducing an equivalence between the rational homotopy category of $\cS_{2}$ and a homotopy category associated 
to \dgl. Denote by $\lambda\colon \cS_{2}\to \dgl$ the Quillen functor. 
A \emph{Quillen model} of $X\in\cS_{2}$ is a
dgl map $\varphi\colon (\L(W),\partial)\to \lambda(X)$, of domain a free dgl, inducing an isomorphism in homology. 
The dgl $(\L(W),\partial)$ itself is also called a Quillen model and if $(\L(W),\partial)$ is minimal, we say a 
\emph{Quillen minimal model} of $X$. The latter is unique up to isomorphisms.

\medskip
The homology groups of a  Quillen  model $(\L(W),\partial)$ of $X$
are isomorphic to the homotopy groups of the loop space
$\Omega X$. They are  equipped with a structure of graded Lie algebra whose bracket is the
Samelson bracket %
$[-,-]\colon \pi_{p}(\Omega X)\otimes \pi_{q}(\Omega X)\to \pi_{p+q}(\Omega X)$.
The homology groups of the  indecomposables, $(W,Q\partial)$, are desuspensions of the rational homology of $X$;
i.e., $H_{i}(W,Q\partial)\cong H_{i+1}(X;\Q)$.

\begin{remark}\label{rem:product}
Some topological constructions can be translated algebraically on models of spaces. 
This is the case for the product of two spaces $X$ and $X'$ 
from their respective algebraic models, $(\L(V),\partial)$ and $(\L(V'),\partial)$.
The product space $X\times X'$ has for Quillen model
$(M,\delta)=(\L(V\oplus V'\oplus s(V\otimes V'), \delta)$ with $\delta v=\partial v$, $\delta v'=\partial v'$,
$\delta(s(v\otimes v'))=[v,v']-\beta(v,v')$ and $|s(v\otimes v')|=|v|+|v'|+1$.

If $v$ and $v'$ are $\partial$-cycles, the element $\beta(v,v')$ is zero. 
We will only consider such elements in what follows but, in the general case, the form of $\beta(v,v')$ can be specified: it is a decomposable element of the kernel of the canonical map $M\to \L(V)\oplus \L(V')$,
see \cite[Proposition VII.1.(2)]{Daniel}.
\end{remark}

\begin{remark}\label{rem:lratr}
  Category \dgl is at the crossroads leading to the categories \cdga of  commutative differential graded  algebras, 
 \cdgc of 2-reduced cocommutative differential graded coalgebras 
 (no element of degree <2) and \dga of differential algebras. 
 Let's recall from \cite{Quillen} the main functors among each other.
$$\xymatrix{
&&\dga
\\
\cdga&
\cdgc\ar[l]_-{\sharp} \ar@<1ex>[r]^-{\cL}
&\dgl\ar[u]_-{\cU}\ar@<1ex>[l]^-{\cC}
}$$

$\bullet$ Functor $\cL$ is the primitive elements of the cobar construction. We will use it only on 
cocommutative  graded coalgebras endowed with a zero differential. 
The explicit corresponding value is given when we use it  in \eqref{equa:dquadratic}. 

$\bullet$ Functor $\cC$ is the right adjoint functor of $\cL$, given by the classical Cartan-Eilenberg-Chevalley construction. We do not use it in this work.

$\bullet$ Functor $\sharp$ is the dualisation. 
We will use it in the reverse way when it exists:   the dual of a 2-reduced   graded algebra of finite type
has a graded commutative coalgebra structure whose image we can take under $\cL$. 

$\bullet$ The enveloping algebra $\cU$ is the left adjoint of the underlying Lie algebra functor from \dga to \dgl\!.
Let  $(L,\partial)$ be a dgl. The differential graded algebra $\cU(L,\partial)$ is constructed as follows.
 We denote by $(T(L),\partial)$ the tensor algebra on $L$ and by $\cJ$ the ideal of $T(L)$
generated by the elements $v\otimes v'-(-1)^{|v|\,|v'|}v'\otimes v -[v,v']$, $v$ and $v'$ in $L$.
The enveloping algebra is the quotient $\cU(L,\partial)=(T(L),\partial)/\cJ$.
 \end{remark}

The next definition produces a  transcription in \dgl of the notion of cone-length for a space.

\begin{definition}\label{def:decomposition}
Let $(\L(V),\partial)$ be  a free dgl.  
\begin{enumerate}
\item A \emph{decomposition of length $k$ of $(\L(V),\partial)$}  
is a decomposition
$V=V_{(1)}\oplus\dots\oplus V_{(k)}$
such that
$\partial_{|V_{(1)}}=0$ and 
$\partial V_{(i)}\subset \L(V_{(1)}\oplus\dots\oplus V_{(i-1)})$,
for any $2\leq i\leq k$. 
The lower index in $V_{(i)}$ is called the \emph{filtered degree.}
\item The \emph{cone-length} of a free dgl $\crL$ is the smallest integer $k$ for which there exists
a quasi-isomorphism
$\crL\xrightarrow{\simeq} (\L(V),\partial)
$,
where $(\L(V),\partial)$ is equipped with a decomposition of length $k$. 
We denote it $\cl \crL$.
\end{enumerate}
\end{definition}

If $\crL$ is a  Quillen model of a simply connected space $X$, 
the cone-length of $\crL$ is equal to the cone-length of $X_{\Q}$, see \cite[Proposition 2.7]{LemaireSigrist}.

\medskip
Let us notice that the existence of a decomposition of length $k$ for the differential $\partial$ 
is not an invariant up to isomorphism (\cite{LemaireSigrist}). 
 For instance,
let $(L,\partial)=(\L(a,b,e,f),\partial)$ with
$|a|=1$, $|b|=3$, $|e|=4$, $|f|=6$ with
$\partial a=\partial e=0$, $\partial b=[a,a]$, $\partial f=[a,e]+[a,[a,b]]$.
A decomposition for $\partial$ is of length 3 but the substitution of generators $e'=e+[a,b]$ brings a 
decomposition of  length~2.

\begin{remark}\label{scheerertanre}
The  inequalities
\begin{equation}\label{equa:catwedge}
\cat X\leq \cat(X\vee S^n)\leq \cl(X\vee S^n)
\end{equation}
are sufficient for the determination of the LS-category of our examples.
Note however that a transcription of   LS-category in \dgl exists in \cite[Section 7.2]{ST1}.
To be complete, mention also the existence of an extension of the representation 
in complete~\dgl of  non-simply connected spaces in \cite{BFMT}.
\end{remark}

\begin{remark}\label{rem:sullivan}
Let $X$ be a simply connected simplicial set, with finite dimensional rational homotopy groups.
An algebraization of $X$ in terms of
commutative differential graded algebras (cdga) is provided by D. Sullivan (\cite{Sullivan})
who also introduces a realization functor transforming a cdga $(A,d_{A})$ 
into a simplicial set $\langle  A,d_{A}\rangle$.
If $A= \land V$ is a free commutative graded algebra over a graded vector space 
$V=\oplus_{i\geq 2}V^i$ with $\dim V^i<\infty$ and if the differential $d_A$ is decomposable, then
there is a dualization between the vector space $V$ and the homotopy groups 
of the realization $\langle \land V,d\rangle$.

\medskip
Transcriptions of LS-category and cone-length of spaces in terms of  Sullivan  models already exist; they
are respectively due to Y. F\'elix and S. Halperin (\cite{FH1}), and O. Cornea (\cite{Octav1}).
Note also that K. Hess has greatly simplified the determination of the LS-category with cdga's (\cite{Kath})
by reducing its characterization to morphisms of ad'hoc module maps.
\end{remark}

  \section{Cell attachments  preserving the  LS-category}\label{sec:examplescat}
   
   \begin{quote}
   In this section, we present a generic case of  cell attachment
  which does not increase the LS-category. 
  \end{quote}
  
  {\sc Construction:}
  Suppose that $L=(\L(W),\partial)$ is a free dgl endowed with a decomposition of length $k$,
  $W=\oplus_{i=1}^k W_{(i)}$, as in \defref{def:decomposition}. 
  Moreover, we also suppose the existence of a cycle
  $\alpha\in\L(W_{(1)})$ such that there exists  $\hat{\alpha}$ 
 with  $\partial\hat{\alpha}=\alpha$ and $\hat{\alpha}\notin \L(W_{\leq (k-1)})$. 
 Let's set $n=|\alpha|$, thus $n+1=|\hat{\alpha}|$.
 As the case $k=2$ is known (\cite{FT1}), we suppose $k\geq 3$.
 
 \medskip
 We consider two copies $L$, $L'$ of this situation and the following dgl,
 \begin{equation}\label{equa:firststep}
 \crL'=(\L(V),\partial)=(L \sqcup L' \sqcup \L(s(W_{(1)}\otimes W'_{(1)}),\partial),
 \end{equation}
  with the previous differential on the free product $L\sqcup L'$ and
  $\partial(s(x\otimes y'))=[x,y']$, if $x\in W_{(1)}$, $y'\in W'_{(1)}$.
  (This is a sub-dgl of the model of a product recalled in \remref{rem:product}.)
  We  set $s(W_{(1)}\otimes W'_{(1)})$ in filtered degree 2.
  By construction, there exists $\gamma\in \L(W_{(1)}\oplus W'_{(1)}\oplus s(W_{(1)}\otimes W'_{(1)}))$
  such that $\partial\gamma=(-1)^{n+1}[\alpha,\alpha']$ and $|\gamma|=2n+1$.

  \medskip
  As the element $[\alpha,\hat{\alpha}']+\gamma$ is a cycle, we can set
  \begin{equation}\label{equa:secondstep}
  \crL=(\L(V\oplus \Q v ),\partial)=(\crL'\sqcup \L(v),\partial)
  \end{equation}
  with $\partial v=[\alpha,\hat{\alpha}']+ \gamma\in \crL'$ and  $|v|=2n+2$.
  With our choice of generators, the new element $v$ gives a decomposition of length $k+1$ for the differential $\partial$.
  Nevertheless, we get the following result.

  \begin{proposition}[{\cite[Proposition 1]{Dupont}}]\label{prop:thecat}
  With the previous notations, there exists a sphere $S^{n+2}$ of model $(\L(w),0)$,
  $|w|=n+1$,  such that
  $$\cl(\crL\sqcup (\L(w),0))\leq k.$$
  Thus, we  have $\cat\crL\leq k$.
  \end{proposition}
  
  \begin{proof}
  We proceed to the following substitution of variables in the dgl
  $\crL\sqcup (\L(w),0)$,
  $$\varphi(w)=w+\hat{\alpha}, \quad 
   \varphi(v)=v-[w+\hat{\alpha},\hat{\alpha}'],$$
  and the identity map on the other generators. 
  Let us denote $w'=\varphi(w)$
   and $v'=\varphi(v)$.
   We obtain a free dgl, $\crM$, isomorphic to $\crL\sqcup \L(w)$, with
   \begin{equation}\label{equa:crLremanie}
   \crM= (\crL'\sqcup \L(v',w'),d),\quad \text{and}
      \end{equation}
$$
    dw'=\alpha, \quad
     dv'=
  [\alpha,\hat{\alpha}']+\gamma -[\alpha,\hat{\alpha}']-(-1)^{n+1}[w+\hat{\alpha},\alpha']
 =
  \gamma +(-1)^n[w',\alpha'].$$
  With this new differential, we can modify the  decomposition of $V\oplus \Q w'\oplus \Q v'$. 
  We keep the previous decomposition for $V$ and set 
  $w'$ in filtration degree $2$, since $dw'\in \L(W_{(1)})$.
  As $\gamma$ and $w'$ are in filtration degree 2, and $\alpha'$  in filtration degree 1, the differential
  $dv'$ is in filtration degree 2 and we can set $v'$ in filtration degree~3.
   We obtain a decomposition of length $k$, as announced.
  \end{proof}
  \section{Construction of the dgl $\crL_{k}$ for $k\geq 3$.}\label{sec:examplescl}
 
 \begin{quote}
 We specify a particular case of  the  pattern developed in  \secref{sec:examplescat}. As
 we have to determine, in particular, some properties of the  automorphisms of the minimal model,
 we minimize the dimension of the  vector space of indecomposables of this model by choosing 
 rational spaces with few rational homology together with a large LS-category.
 Let $k\geq 3$ be a fixed integer, the case $k=3$ corresponding to Dupont's example.
 As $k$ is fixed, we simplify the notation by writing 
 $\crL$ instead of $\crL_{k}$.
 \end{quote}

 Let $A$ be an algebra of finite type. Its linear dual, $\sharp A$, is canonically endowed with a structure
 of cocommutative coalgebra, with zero differential.
 The  Quillen functor $\cL$ (\cite{Quillen} and \remref{rem:lratr})  
 gives a free dgl $\cL(\sharp A,0)$ on the desuspension of 
  $\sharp A$, $(s^{-1}\sharp A)_{p}=(\sharp A)_{p+1}$. 
  The differential $d$ of $\cL(\sharp A,0)$ is quadratic and defined 
 by the pairing between $A$ and $\sharp A$, together with the  law of algebra.
 More explicitely, we have  (\cite[I.1.(7)]{Daniel})
 \begin{equation}\label{equa:dquadratic}
 \langle sa_{1},sa_{2}; ds^{-1}b\rangle=(-1)^{|a_{1}|}\langle a_{1}a_{2};b\rangle, \quad a_{i}\in A,\; b\in\sharp A.
 \end{equation}
 We namely consider the two algebras,
 \begin{equation}\label{equa:twoalgebras}
 A_{k}=(\land u)/u^{k+1}
 \quad \text{and}\quad
 B_{k}=(\land v)/v^3, 
 \quad \text{with}\quad
  |u|=4, \;|v|=2k.
 \end{equation}
 We get
 $$\cL(A_{k})=(\L(a_{1},\dots,a_{k}),d)
 \quad\text{and}\quad
 \cL(B_{k})=(\L(b_{1},b_{2}),d),
 $$
 with $ |a_{i}|=4i-1$ and $ |b_{i}|=2ki-1$. 
 To simplify the notation we set $a=a_{1}$ and $b=b_{1}$.
The  differentials $d$ are perfectly defined by \eqref{equa:dquadratic}. In particular, we have
 $$db=da=0, \quad
 db_{2}=[b,b],  \quad
 da_{2}=[a,a], \quad
 da_{3}=[a,a_{2}], \quad
 da_{4}=[a_{2},a_{2}]+4[a,a_{3}],\quad
 \dots$$
 With reference to the CW-structure of the connected sum of 
 two compact connected oriented  $4k$-dimensional manifolds,
  we proceed to a ``connected sum''  of these two dgl's to obtain the dgl
 \begin{equation}\label{equa:theL}
 (L,\partial)= (\L(a,a_{2,}\dots,a_{k-1},b,c),\partial)
 \end{equation}
with $|c|=4k-1$, $\partial c= -[b,b]-da_{k}$ and $\partial=d$ on the other elements.
Let us emphasize that $da_{k}\in \L^{[2]}(a,\dots,a_{k-1})$ and 
$da_{k}\notin \L(a,\dots,a_{k-2})$.
We thus have a decomposition of length $k$
for the differential $\partial$ with $L=\L(W)$ and
\begin{equation}\label{equa:decompositionkW}
W_{(1)}=\Q(a,b),\quad 
W_{(2)}=\Q\, a_{2}, \quad \dots,\quad
W_{(k-1)}=\Q \,a_{k-1},\quad  
W_{(k)}=\Q \,c.
\end{equation}
Following the notation of \secref{sec:examplescat}, we introduce  $\alpha=[a,[b,b]]$.  Let us note that
the element $[a,da_{k}]$ is a cycle of degree $3+4k-2=4k+1$ in $\L(a,a_{2},\dots,a_{k-1})$.
Let  $\langle A_{k-1}\rangle$ be the simplicial set associated to $A_{k-1}$ by the Sullivan realization functor
and $\Omega\langle A_{k-1}\rangle$ its image by the loop space functor.  The
homotopy vector space $\pi_{*}\Omega \langle A_{k-1}\rangle\otimes \Q$ is concentrated in degrees 3 and $4k-2$
which correspond to the desuspension of the generators $u,\,u'$, $du=0$, $du'=u^k$,
$|u|=4$, $|u'|=4k-1$ of its Sullivan model, see \remref{rem:sullivan}. 
So, the cycle $[a,da_{k}]\in L$ is a boundary in $(L,\partial)$ and there exists 
$f\in \L^{[2]}(a,a_{2},\dots,a_{k-1})$,
of degree $|f|=4k+2$ such that $\partial f=[a,da_{k}]$.
Thus, the element
$\hat{\alpha}=[a,c]-f$ verifies 
$|\hat{\alpha}|=4k+2$,
$\hat{\alpha}\notin\L(W_{\leq k-1})$ and
\begin{equation}\label{equa:delhatalpha}
\partial \hat{\alpha}
=
-[a,\partial c]-\partial f
=
[a,[b,b]]+[a,da_{k}]-\partial f
=
 [a,[b,b]]=\alpha.
\end{equation}
Also, by definition of $\hat{\alpha}$, we have
$\partial [a,c]=\alpha+\partial f$. We get a dgl $(L,\partial)$ as that of \secref{sec:examplescat}
and  follow the procedure described in \eqref{equa:firststep}: 
we take a wedge of  two copies $L$, $L'$ and kill the brackets between the spherical homology classes
by setting
 $$\crL'=(\L(V),\partial)=(L\sqcup L'
 \sqcup \L(s(a\otimes a'), s(a\otimes b'), s(b\otimes a'), s(b\otimes b')),\partial),
 $$
 with 
 $\partial s(a\otimes a')=[a,a']$,
  $\partial s(a\otimes b')=[a,b']$,
   $\partial s(b\otimes a')=[b,a']$,
    $\partial s(b\otimes b')=[b,b']$.
 The dgl $\crL'$ has a decomposition of length $k$ for the differential $\partial$, with
 \begin{equation}\label{equa:decompositionkV}
\left\{
\begin{array}{ll}
V_{(1)}=\Q(a,a',b,b'),& 
V_{(2)}=\Q( a_{2},a'_{2},s(a\otimes a'), s(a\otimes b'), s(b\otimes a'), s(b\otimes b')),\\
\dots \\
V_{(k-1)}=\Q(a_{k-1},a'_{k-1}),&  
V_{(k)}=\Q(c,c').
\end{array}
\right.
\end{equation}
In this situation the integer $n$ of \secref{sec:examplescat} is an odd number, since
$n=|\alpha|=|[a,[b,b]]|=3+4k-2=4k+1$. Therefore, the element $\gamma\in \L^{[5]}(V_{\leq (2)})$ verifies
$$\partial \gamma=[\alpha,\alpha']=[[a,[b,b]],[a',[b',b']]].$$
We continue with the construction done in \eqref{equa:secondstep}:
we add a new generator $v$, $|v|=8k+4$, to get
\begin{equation}\label{equa:theL}
\crL=\crL' \sqcup \L(v) =(\L(V\oplus \Q\,v),\partial).
\end{equation}
The differential $\partial$ on $\crL'$ is defined above and we set
$$\partial v=[\alpha,\hat{\alpha}']+\gamma=
[[a,[b,b]],[a',c']-f']+\gamma
\in\L^{[5]}(V_{\leq (k)}).
$$
We have a decomposition of length $(k+1)$ of $\crL$ for the differential $\partial$ by adding to 
\eqref{equa:decompositionkV} the vector space $\Q\,v$ in filtration $k+1$.
The Main Theorem is a consequence of  the following statement.
 
 \begin{theoremb}\label{thm:superDupont}
 The dgl $\crL$ verifies
 $$\cat\,\crL = k
 \quad \text{and}\quad
 \cl \crL=k+1.
 $$
 \end{theoremb}
 
 The first assertion is a direct corollary of \propref{prop:thecat}. We have to prove that there is no free dgl
 $\crM=(\L(Z),\delta)$, quasi-isomorphic to $\crL$ and endowed with
 a decomposition of length $k$ for the differential $\delta$. 
 As first step, in the next section, we consider for $\crM$ the dgl $\crL$ itself.
 
 \section{Non-existence of isomorphisms of $\crL_{k}$  shortening the length decomposition}\label{sec:iso}
 
 \begin{quote}
 In this section, we show that there is no isomorphism of $\crL_{k}$ shortening the length decomposition,
 as a consequence of the following statement.
 As before $k\geq 3$ is a fixed integer and we denote $\crL_{k}$ by $\crL$.
 \end{quote}

 \begin{proposition}\label{prop:paskminimal}
 There is no decomposable element $x$ such that 
 $\partial(v-x)\in \L(V_{(1)}\oplus\dots\oplus V_{(k-1)})$.
 \end{proposition}
 
 If such $x$ exists, the substitution of generators, $v\mapsto v-x$,
  should give a decomposition of length $k$ for $\crL$.
 
 \begin{proof}
 For any decomposable element $x$, the differential $\partial (v-x)$ belongs to
 $\L(V_{\leq (k)})$. 
 We are thus working in $\crL'$ whose differential is \emph{quadratic.}
We make a proof by contradiction and assume that such an element $x$ exists
with $\partial(v-x)\in \L( V_{\leq (k-1)})$. Thus its differential verifies
\begin{equation}\label{equa:delomega}
\partial x=[\alpha,[a',c']]+\beta=
[[a,[b,b]],[a',c']]+\beta, \quad \text{with}\quad\beta\in \L(V_{\leq (k-1)}).
\end{equation}
Let us first note that the bracket $[[a,[b,b]],[a',c']]$ can only appear as the differential of brackets of length 4,
formed by the elements 
\begin{equation}\label{equa:laliste}
\{a,c,a',c'\},\quad \{s(a\otimes a'), b,b,c'\}\quad \text{or}\quad \{a, s(b\otimes a'),b,c'\}.
\end{equation}
To simplify the situation, we first replace $\crL'$ by the   dgl
 \begin{equation}\label{equa:lbar}
 {\crL}'_{1}=(\left(\L(a,b)\oplus \L(a',b')\right) \sqcup \L(a_{2},\dots,a_{k-1},c,a',\dots,a'_{k-1},c'),{\partial}).
 \end{equation}
 This is the quotient of $\crL'$ by the  ideal generated by the elements $s(-\otimes -')$ and their differential. 
 We keep the notation $\partial$ for the induced differential on $\crL'_{1}$ and
denote by ${\rho_{1}}\colon \crL'\to{\crL'_{1}}$ the canonical map.

 \medskip
 With \eqref{equa:laliste}, the element $\rho_{1}([[a,[b,b]],[a',c']])$ appears only in the differential of 
 images by $\rho_{1}$ of elements of bracket length 4 of the distinct elements,
 $\{a,c,a',c'\}$.
 In $\crL'$, they are the multilinear Lie polynomials of $\{a,\, c=t_{1},\, a'= t_{2},\, c'=t_{3}\}$ 
 of \cite[Section 5.6.2]{Lie}.
They  form a subvector space of dimension 6 of $\crL'$ with basis the set
 $$B=\left\{[t_{\sigma(1)},[t_{\sigma(2)},[t_{\sigma(3)},a]]]\mid \sigma\in\cS_{3}\right\},
 $$
 where $\cS_{3}$ is the symmetric group of permutations of 3 elements. 
 The  image  of $B$ by ${\rho_{1}}$ is in bijection with the elements of $B$ such that $t_{\sigma(3)}\neq a'$,
 therefore it is of cardinality 4.
 To reduce the complexity again, we quotient $\crL'_{1}$  by the differential ideal generated by $[c',a]$. 
 (So we quotient also by the elements $[a,a'_{i}]$.)
 We also kill the ideal generated by $[a',a']$, a reduction that will be used hereafter.
 We thus obtain a dgl  $(\crL'_{2},\partial)$
 and denote by $\rho_{2}\colon \crL'\to \crL'_{2}$ the canonical surjection.
 If $\alpha\in\cU\crL'$, we denote $\ov \alpha\in\cU\crL'_{2}$ its image by $\cU\rho_{2}$,
 where $\cU$ is the universal enveloping algebra functor, see \remref{rem:lratr}.
 We also denote by $\rho_{2}\colon \crL'\to \cU\crL'_{2}$ the composition of $\rho_{2}$ with the canonical injection
 $\crL'_{2}\to \cU\crL'_{2}$.
  To sum up, the algebra $\crL'_{2}$ is the quotient of $\crL'$ by the ideal generated by
  \begin{equation}\label{equa:quotient}
    \small
\left\{
s(a\otimes a'), s(a\otimes b'), s(b\otimes a'), s(b\otimes b'),
[a,a'], [a,b'], [b,a'], [b,b'],[a,c'], [a,a'_{i}], [a',a']
\right\},
\end{equation}
  for all $i$, $2\leq i\leq k-1$.
  The image of the basis $B$ by $\rho_{2}$ is of cardinality 2 and corresponds to $t_{\sigma(3)}=c$. 
 More specifically, we are reduced to the image by $\rho_{2}$ of
 $$
 [c',[a',[a,c]]]\quad \text{and}\quad
 [a',[c',[a,c]]].
 $$
  From  antisymmetry and Jacobi identity, we deduce
  $$[c,a]=[a,c] \quad \text{and} \quad
  [[a',c'],[a,c]]=[a',[c',[a,c]]] + [c',[a',[a,c]]].$$
 So, we can  choose  $\left\{\rho_{2}( [[a,c],[a',c']]), \rho_{2}([c',[a',[a,c]]])\right\}$ as basis  of $\rho_{2}(B)$.
 In this preamble, we established the existence of two rational numbers $\lambda_{1}$, $\lambda_{2}$ 
 and of an element $y\in\crL'$ 
 whose differential ${\partial}{y}$ does not contain a component in ${[[a,[b,b]],[a',c']]}$, such that
 \begin{equation}\label{equa:rho2x}
 \rho_{2}(x)=
 \rho_{2}\left(
 \lambda_{1}{[[a,c],[a',c']]}+
 \lambda_{2}{[c',[a',[a,c]]]}+{y}
 \right).
 \end{equation}
The image by $\rho_{2}$ of the equality \eqref{equa:delomega} becomes
 \begin{equation}\label{equa:clerho2}
\rho_{2}\left( \lambda_{1} {\partial}{[[a,c],[a',c']]}+
 \lambda_{2} {\partial}{[c',[a',[a,c]]]}+\partial {y}-
 {[[a,[b,b]],[a',c']]}-{\beta}\right)=0,  
 \end{equation}
with $\beta\in \L(V_{\leq (k-1)})$.

\medskip
$\bullet$ \emph{Suppose  $\lambda_{1}\neq 1$.}
 The element 
$\rho_{2}\left([[a,[b,b]],[a',c']]\right)$ has a component in $\ov{ab^2c'a'}\in\cU\crL'_{2}$ which is not killed by 
the relations \eqref{equa:quotient}. Thus, we have
\begin{equation}\label{equa:rho2neq0}
\rho_{2}\left([[a,[b,b]],[a',c']]\right)\neq 0.
\end{equation} 
Then, if  $\lambda_{2}=0$, a component of
$\rho_{2}\left([[a,[b,b]],[a',c']]\right)$ would belong to 
$\rho_{2}(\L(V_{\leq (k-1)}))$, which is a contradiction. 
We  conclude:  $\lambda_{2}\neq 0$.

\medskip
We focus on the component $w=\ov{a'a'_{k-1}a'ca}\in \cU\crL'_{2}$ of
$\rho_{2}([[a',a'_{k-1}],[a',[a',c]]])$ which appears in
$\rho_{2}(\partial [c',[a',[a,c]]])$. 
 We first note that $w\neq 0$ in $\cU\crL'_{2}$ and study its occurence in the other terms of
\eqref{equa:clerho2}.
\begin{enumerate}[(i)]
\item In the expansion of
$\rho_{2}(\partial[[a,c],[a',c']])$,  the component $w$ can only appear in the term
$\rho_{2}([[a,c],[a',[a',a'_{k-1}]])$.
But
$2[a',[a',a'_{k-1}]]= [a'_{k-1},[a',a']]$  
and the component in $w$
of $\rho_{2}(\partial[[a,c],[a',c']])$ is zero since $[a',a']$ has been killed.
\item We are left with $\rho_{2}(\partial y)$.
From the expression of the (quadratic) differential  $\partial$,  the element $w$ can only 
appear in the image by $\rho_{2}\partial$ of a 
multilinear   Lie polynomial of $\{a,c,a',c'\}$, $\{a'_{2},a'_{k-1},a,c\}$ or $\{s(a\otimes a'),a'_{k-1},a',c\}$. 
 The first case is excluded by choice of $y$.
  In  $\cU\crL'_{2}$, the elements $\partial a'_{2}=[a',a']$ and $s(a\otimes a')$ are killed. 
  Thus $\rho_{2}({\partial}\,{y})$ has
 no component in $w$.
\end{enumerate}
We conclude that $w=\ov{a'a'_{k-1}a'ca}$ appears as an element of $\cU \rho_{2}(\L(V_{\leq (k-1)}))$.
But $c$ corresponds to a cup product of length $k$ and $\L(V_{\leq (k-1)})$ is of cone-length $k-1$. 
We get a contradiction and the proposition is established.

\medskip $\bullet$
\emph{Suppose $\lambda_{1}=1$.} We justify  the following claim after the proof.\\
{\sc Claim:} \emph{there exists a decomposable element $\tx$, having a zero component in  $[[a,c],[a',c']]$,
 and such that
 $\partial \tx=[[a,c],\alpha']+\beta'$ with $\beta'\in \L(V_{\leq (k-1)})$.}
 
 \medskip
 Repeating the argument from the beginning of this proof with an exchange of $L$ and $L'$, 
 \eqref{equa:clerho2} becomes
 \begin{equation}\label{equa:exchange}
\rho_{2}\left(  
 \lambda_{2} {\partial}{[c,[a,[a',c']]]}+\partial {y}-{[[a,c],[a',[b',b']]}-{\beta}\right)=0,
 \end{equation}
with $\beta\in \L(V_{\leq (k-1)})$.
The same reasoning than above gives the contradiction.
 \end{proof}
   
\begin{proof}[Proof of the claim]
Suppose 
$x= [[a,c],[a',c']] + y$ where $y$ has a zero component 
in $[[a,c],[a',c']]$. We set
$\ttx =y +[f,[a',c']]+[[a,c],f']$. 
The following computation gives the claim.
\begin{eqnarray*}
\partial \ttx&=&
\partial x - \partial( [[a,c],[a',c']]) +\partial([f,[a',c']])+
\partial([[a,c],f'])+\beta\\
&=&
[\alpha,[a',c']] - [\alpha,[a',c']] - [\partial f, [a',c']]
-[[a,c],\alpha']-[[a,c],\partial f']\\
&&
+[\partial f,[a',c']]+[f,\alpha']+[f,\partial f']
+[\alpha,f']+[\partial f,f']+[[a,c],\partial f']
+\partial \beta\\
&=&
-[[a,c],\alpha']+\beta',
\end{eqnarray*}
with $\beta,\,\beta'\in \L(V_{\leq (k-1)})$ and $\tx=-\ttx$.
  \end{proof}
  
  \section{Determination of the cone-length of $\crL_{k}$}\label{sec:proof}
  \begin{quote}
  In this section, we prove $\cl \crL_{k}= k+1$ and the Main Theorem is  a consequence of this determination.
  As $k\geq 3$ is a fixed integer, we denote $\crL_{k}$ by $\crL$ in the following.
  \end{quote}
  
  Let $\crL=(\L(V\oplus \Q\,v),\partial)$ be the minimal dgl built in \secref{sec:examplescl}.
  We already know that $\cat\crL=k$, so its cone-length is equal to $k$ or $k+1$. 
  The proof that $\cl \crL=k+1$ is made by contradiction assuming $\cl\crL=k$.
  So, suppose that there is a quasi-isomorphism,
  \begin{equation}\label{equa:mininonmini}
  \varphi\colon \crL=(\L(V\oplus \Q\,v),\partial)\xrightarrow{\simeq} (\L(W),\delta),
  \end{equation}
  where $(\L(W),\delta)\in \dgl$ admits a decomposition of length $k$, $W=\oplus_{i=1}^k W_{(i)}$.
  The differential $\delta$  is not assumed to be decomposable and we denote by 
  $\delta_{1}\colon W_{\ast}\to W_{\ast-1}$ its linear part.
  
  \medskip
  The use of  \propref{prop:paskminimal} requires  a minimal model. 
  A first approach is the elimination of the linear part $\delta_{1}$ of the differential, 
  replacing $(\L(W),\delta)$ by its minimal model, but we
  must also take into account the decomposition in length $k$ of $(\L(W),\delta)$. 
  It is quite easy to show that we can replace $(\L(W),\delta)$ by a dgl having such a decomposition 
  and whose linear part of the differential satisfies 
  $ \delta_{1}W_{(i)}\subset W_{<(i-1)}$. 
  (We do not detail this point which we do not use subsequently.)
 But the elimination of the other components of the linear part of the differential 
   while keeping a decomposition of length $k$  is generally impossible. 
   Indeed, let us consider for example
    $W=W_{(1)}\oplus W_{(2)}\oplus W_{(3)}$. As announced above,  we can assume 
   $\delta_{1} W_{(2)}=0$ and $\delta_{1}W_{(3)}\subset W_{(1)}$.
  To eliminate the entire linear part $\delta_{1}$, 
   we are now led to identify elements, $w=\delta_{1}x\in W_{(1)}$ 
   with $x\in W_{(3)}$, to $\delta x\in \L(W_{(1)}\oplus W_{(2)})$. 
   The filtration is not respected and the quotient is no longer provided with a decomposition. 

  \medskip
  
Although the differential of the minimal model does not respect filtration, 
there are still interactions between them.
First, the image of the differential remains included in the subspace of filtration elements strictly less than $k$. 
Secondly, following the way of \cite{Dupont} for $k=3$, we show  that $(\L(W),\delta)$ 
gives rise to a minimal dgl with quadratic differential, equipped with a decomposition of length $k$, for any $k$. 
(Naturally this dgl is not quasi isomorphic to $(\L(W),\delta)$.)
All that  put us in a situation where  \propref{prop:paskminimal} can give information
and this will be enough to reveal a contradiction.
The following statement specifies the two previous properties.

\begin{proposition}\label{prop:quadraticdiff}
Let $(\L(\ov K),\delta)$ be the minimal  dgl associated to $(\L(W),\delta)$,
seen as a differential Lie subalgebra of $(\L(W),\delta)$. 
The  following properties are satisfied.
\begin{enumerate}[1)]
\item Let $\delta_{2}$ be the quadratic part of $\delta$ on $\L(\ov K)$. 
Then, the minimal  dgl,  $(\L(\ov K),\delta_{2})$, has a decomposition of length $k$.
\item For this decomposition, the differential $\delta$ verifies $\delta \ov K\subset \L(\ov K_{<(k)})$.
\end{enumerate}
\end{proposition}

\begin{proof}
To accomplish this, let's  recall (\cite[Proposition 3.18]{BFMT}) 
the construction of the minimal dgl model of $(\L(W),\delta)$,
detailing more particularly the quadratic part of $\delta$.
We write $W=K\oplus (R\oplus \delta_{1}R)$ with $K\subset \ker\delta_{1}$.
We set  filtrations on $K$ and $R$ by $K_{(i)}=K\cap W_{(i)}$ and $R_{(i)}=R\cap W_{(i)}$. 

\medskip
In the following short exact sequence, the map $\rho$ is induced
by the projection $W\to K$ and the kernel $\crI$ of $\rho$ is acyclic 
for the differential $\delta_{1}$,
induced by the linear part  of  
$\delta$ on $\L(W)$:
$$\xymatrix@1{
0\ar[r]&
(\crI,\delta_{1})\ar[r]&
(\L(W),\delta_{1})\ar[r]^-{\rho}&
(\L(K),0)\ar[r]&
0.}$$
We filter the ideal $\crI$ by $\crI^{[q]}=\crI\cap\L^{[q]}(W)$. Let $z\in K_{(i)}$. We decompose the quadratic part of $\delta z$ as
\begin{equation}\label{equa:quad}
\delta_{2}z=\alpha(z)+\beta(z)\in \L^{[2]}(K)\oplus \crI^{[2]}.
\end{equation}
As $(\L(W),\delta)$ has a decomposition of length $k$,  we have 
\begin{equation}\label{equa:filtrationalpha}
\alpha(z)\in \L(K_{<(i)}).
\end{equation}
From $\delta^2=0$ and \eqref{equa:quad}, we deduce
$\delta_{1}\delta_{2}z
=
-\delta_{2}\delta_{1}z=0$ and
$\delta_{1}\delta_{2}z
=
\delta_{1}\beta(z)$.
The kernel $\crI$ being acyclic, there exists $\mu(z)\in\crI^{[2]}$ such that
$\beta(z)=\delta_{1}\mu(z)$. We set $z'=z-\mu(z)$. The linear and quadratic parts of $\delta z'$ in $\L(W)$ are
\begin{eqnarray}\label{equa:quad2}
\delta_{1}z'=0
\quad\text{and}\quad
\delta_{2}z'=\delta_{2}z-\delta_{1}\mu(z)=\alpha(z) \in \L^{[2]}(K).
\end{eqnarray} 
We carry out this transformation for each element $z_{\ell}$ of a basis $(z_{\ell})_{\ell\in I}$ of $K$ and  denote by $K'$ 
the vector space of basis   $(z'_{\ell})_{\ell\in I}$.
As $\crI$ is a Lie ideal, the replacement of $z_{\ell}$ by its value gives
$\alpha(z_{\ell})=\alpha(z'_{\ell})+\crI^{>[2]}$. We thus have
\begin{equation}\label{equa:prime}
\delta_{2}z'_{\ell}\in\alpha(z'_{\ell})+\crI^{>[2]}.
\end{equation}

\medskip
By induction on the size of brackets in $\L(W)$, a vector space $\ov K$, of basis
$(\ov z_{\ell})_{\ell \in I}$, is constructed in \cite[Proposition 3.18]{BFMT} such that
\begin{enumerate}[(i)]
\item $\ov z_{\ell}-z'_{\ell}\in \crI^{[>2]}$,
%
$\delta (\ov z_{\ell})\in \L(\ov{K})$, and
\item  the canonical injection
$(\L(\ov K),\delta)\sqcup \L(R\oplus \delta R)
\xrightarrow{\cong}
 (\L(W),\delta)$
is an isomorphism.
\end{enumerate}

\medskip
1) 
So $(\L(\ov K),\delta)$ is the minimal model of $(\L(W),\delta)$ and,
with  (i) and \eqref{equa:prime},  the quadratic part of its differential is obtained  by:
\begin{equation}\label{equa:filtrationdeux}
\delta_{2}(\ov z_{\ell})=\alpha(\ov z_{\ell}).
\end{equation}
We define a filtration on $\L(\ov K)$ by assigning to $\ov z_{\ell}$ the filtration degree of $z_{\ell}$.
From \eqref{equa:filtrationalpha} and \eqref{equa:filtrationdeux}, we deduce 
a decomposition of length $k$ on $(\L(\ov K),\delta_{2})$.

\medskip
2) We decompose $R=R_{(k)}\oplus U$ and set $Z=K\oplus U\oplus \delta_{1}R$.
We  observe that $(\L(Z),\delta)$ is a sub dgl of $(\L(W),\delta)$. 
We apply the previous construction of minimal model  to $(\L(Z),\delta)$.
The initial decomposition becomes
$$Z= (K\oplus \delta_{1}R_{(k)})\oplus (U\oplus \delta_{1}U),$$
since the elements of $\delta_{1}R_{(k)}$ are no longer in the image of $\delta_{1}$ in $Z$.
In this process, property (ii) above  
gives an isomorphism of dgl's,
\begin{equation}\label{equa:sanslek}
(\L(\ov{K}\oplus \ov{\delta_{1}R_{(k)}}),\delta)\sqcup \L(U\oplus \delta U)
\xrightarrow{\cong} (\L(Z),\delta).
\end{equation}
Following the construction  by induction of the minimal model, we note 
(see \remref{rem:minimal} for more details) that, if $z\in R_{(k)}$
then in \eqref{equa:sanslek}, we have 
\begin{equation}\label{equa:bardelta1}
\ov{\delta_{1}z}=\delta z.
\end{equation}
In the next diagram, 
the columns are cofibrations and the lines are canonical injections,
$$\xymatrix{
(\L(\delta R_{(k)}),0)\ar@{=}[r]\ar[d]&
(\L(\delta R_{(k)}),0)\ar[d]\\
(\L(\ov K\oplus \delta R_{(k)}),\delta)\ar[r]^-{\simeq}_-{\iota_{1}}\ar[d]&
(\L(Z),\delta)\ar[d]\\
(\L(\ov K\oplus \delta R_{(k)}\oplus R_{k}),\delta)\ar[r]^-{\simeq}_-{\iota_{2}}
&
(\L(Z\oplus R_{(k)}),\delta)\cong \L(W),\delta).
}$$
The map $\iota_{1}$ is a quasi-isomorphism by construction, so is the induced map $\iota_{2}$ between the cofibres.
Therefore the minimal model of $(\L(W),\delta)$ can be obtained from 
$(\L(\ov K\oplus \delta R_{(k)}),\delta)$
by quotienting the ideal generated by  the subspace   $\delta R_{(k)}$, formed by $\delta$-cycles.

\medskip
The dgl $(\L(\ov K\oplus \delta R_{(k)}),\delta)$ is constructed through a sequence of isomorphisms beginning with
$z\mapsto z'=z-\mu(z)$, as described above.
We now use the existence of a decomposition of length $k$ on $(\L(W),\delta)$
and set
$$Z_{<(k)}:=
(K_{<(k)}\oplus \delta_{1}R_{(k)})\oplus (U\oplus \delta_{1}U)\subset W_{(k)}.$$
The Lie subalgebra $\L(Z_{<(k)})$ is stable to the differential $\delta$. 
Let $z\in K_{(k)}$. As in \eqref{equa:quad}, we decompose
$\delta z=\alpha(z)+\beta(z)$,
with $\alpha(z)\in \L(\ov K_{<(k)}\oplus \delta_{1} R_{(k)})$ and $\beta(z)$
 in the kernel $\crK$ of the canonical projection:
$$\xymatrix@1{
0\ar[r]&
(\crK,\delta_{1})\ar[r]&
(\L(Z_{<(k)}),\delta_{1})\ar[r]&
(\L(K_{<(k)}\oplus \delta_{1}R_{(k}),0)\ar[r]&
0.}$$
As Lie algebra (see \cite[Proposition VI.2.(7)]{Daniel} for instance), this kernel $\crK$ is isomorphic to
$$\crK\cong \L((T(K_{<(k)}\oplus \delta_{1}R_{(k)}))\otimes (U\oplus \delta_{1}U)).$$
The dgl $(\crK,\delta_{1})$ is acyclic and the element $\mu(z)$
such that $\delta_{1}\mu(z=\beta(z)$ can be chosen in $\crK$.
There is no element of filtration degree $k$ in $\crK$, so we have
$\delta z'\in \L(Z'_{<(k)})$.
Repetitions of this argument for the determination of a basis $(\ov z)_{\ell\in I}$ prove the result.
\end{proof}
   
 \begin{remark}\label{rem:minimal}
 In the previous proof, the fact that $\ov{\delta_{1}z}=\delta z$ in \eqref{equa:bardelta1} is not entirely obvious. 
 There is a hidden point in the proof of \cite[Proposition 3.18]{BFMT}: 
 the short exact sequences allowing the induction step change at each index $n$.
 Moreover the modifications of generators are not linear, so the notion of bracket length has to be taken with care.
 Let us illustrate this for $n=2$ and  $n=3$  with an element $\delta_{1 }z\in \delta_{1}R_{k}\subset Z$, 
 with the notations of the proof.
 We start with $\delta_{2}\delta_{1}z=-\delta_{1}\delta_{2}z$, so 
 $\alpha(\delta_{1}z)=0$ and $\beta(\delta_{1}z)=-\delta_{2}z$, with gives
 $(\delta_{1}z)'=(\delta_{1}+\delta_{2})z$. 
 For the next step, we work with $\L((K\oplus \delta_{1}R_{k})')$, where $\delta_{1}z+\delta_{2}z$ is an indecomposable. 
 For instance, the quadratic part $\delta'_{2}$ of $(\delta_{1}z)'$ in $\L((K\oplus \delta_{1}R_{k})')$ 
 can be computed with the bracket length in $\L(K\oplus \delta_{1}R_{k})$ as
 $\delta'_{2}(\delta_{1}z)'=\delta_{2}\delta_{1}z+ \delta_{1}\delta_{2}z=0$.
 The next inductive step begins with the writing of the part of bracket length 3, $\delta'_{3}$, of the differential in  
 $\L((K\oplus \delta_{1}R_{k})')$. We determine it from the bracket length in $\L(K\oplus \delta_{1}R_{k})$ as
 $$\delta'_{3}(\delta_{1}z)'=\delta_{3}\delta_{1}z\oplus \delta_{2}\delta_{2}z=-\delta_{1}\delta_{3}(z).$$
 So, the modification of generators is $(\delta_{1}+\delta_{2})z\mapsto (\delta_{1}+\delta_{2}+\delta_{3})z$.
 The rest is straightforward.
 (Let us also mention that the  summation limits in the formula of $d_{1}d_{n+1}z$ located in the middle of the proof of
 \cite[Proposition 3.18]{BFMT} have to be modified in $\sum_{i=2}^{n+1}$.)
 \end{remark}
   
  \begin{proof}[Proof of \thmref{thm:superDupont}]
Let $\crL=(\L(V\oplus \Q\,v),\partial)$ be the minimal dgl built in \secref{sec:examplescl}.
Suppose  there is a quasi-isomorphism,
$
  \varphi\colon \crL \xrightarrow{\simeq} (\L(W),\delta),
$
  where $(\L(W),\delta)\in \dgl$ admits a decomposition of length $k$, $W=\oplus_{i=1}^k W_{(i)}$.
  We are looking for a contradiction.
  
 \medskip
 Recall the existence of quasi-isomorphisms,
 \begin{equation}
 \xymatrix{
 \crL\ar[d]^-{\varphi}
 \ar[rd]^-{\Phi}
 &\\
 (\L(W),\delta)
 \ar@<1ex>[r]^-{\rho}
 & (\L(\ov K),\delta)\ar@<1ex>[l]^-{\iota}
 }
 \end{equation}
 with $\Phi=\rho\circ\varphi$. 
 Denote by $\Phi^{[1]}\colon (\L(V),\partial_{2})\sqcup (\L(\Q v),0)\to (\L(\ov K),\delta_{2})$ the dgl map 
 defined from the linear part $\Phi_{l}\colon V\oplus \Q v\to \ov K$ of $\Phi$. (Recall that the restriction of $\partial$
 to $V$ is quadratic.)
 The map $\Phi$ is a quasi-isomorphism between minimal dgl's, therefore $\Phi$ and $\Phi_{l}$ are isomorphisms.
 Thus there exists $\omega\in\crL$ such that
 \begin{equation}
 \Phi(\omega)=\Phi(v)-\Phi_{l}(v)
 \end{equation}
 and, with \propref{prop:quadraticdiff}, we have
 \begin{equation}\label{equa:diffomega}
 \Phi\partial(v-\omega)=\delta\Phi(v-\omega)=\delta \Phi_{l}(v)\in \L(\ov K_{<(k)}).
 \end{equation}
 By construction, the element $\omega$ is  of degree $8k+4$. 
 The elements of a homogeneous basis of $V\oplus \Q v$ are in degrees
 $4i-1$, with $i=1,\dots,k$, and $2k+1$. 
 These numbers being odd, the element $\omega$ is formed by  brackets of even length. With the respective
 degrees, the element $\omega$ cannot  contained  brackets of length 2, thus $\omega\in\crL^{[\geq 4]}$.
 Denote by $\Omega$ the part of $\omega$ of length 4 
 From \eqref{equa:diffomega}, we deduce
$\Phi_{l}\delta (v-\Omega)\in \L(\ov K_{<(k)})$, since it is the term in the lower bracket length.
\propref{prop:paskminimal} implies that in the linear decomposition of $\delta (v-\Omega)$, 
there must be a bracket of length 5 in which one of the 5 vectors belongs to
the vector space $\Q\{c,c'\}$ generated by $c$ and $c'$. Denote this element by $\varepsilon$.
 The quadratic part of the differential corresponds, by duality (see \cite[Proposition III.3.(9)]{Daniel} for instance) 
 to a cup product.
 In this duality, the elements  $c$ and $c'$ give non zero cup products of length $k$. 
 The image of $\varepsilon$ by a linear isomorphism
 cannot be in a dgl with a decomposition of length $k-1$. 
 This algebraic argument is the traduction of 
 ``the cup length is less than, or equal to, the cone-length.''
 We have got the contradiction.
    \end{proof} 
    
    \begin{remark}\label{rem:recap}
   Below we summarize the characteristics of the main elements introduced in the text.
 \begin{equation*}
 \left\{
 \begin{split}
 &|a_{i}|=4i-1,\quad
 |b|=2k-1,\quad 
 |\alpha|=4k+1,\quad
 |f|=|\hat{\alpha}]=4k+2,\\
 &\alpha=
 [a,[b,b]],\quad 
  \hat{\alpha}=[a,c]-f,\quad
  \partial \hat{\alpha}=\alpha,\quad
  \partial f=[a,da_{k}],\quad \partial \gamma=[\alpha,\alpha'],
 \\
 & |c|=4k-1,\quad
 \partial c=-[b,b]-da_{k},\quad \partial [a,c]=\alpha+\partial f,
 \\
 &
|v|=8k+4, \quad \partial v = [\alpha,\hat{\alpha}']+\gamma =[[a,[b,b]],[a',c']-f']+\gamma.
 \end{split}
 \right.
 \end{equation*}
    \end{remark}

\end{document}